\documentclass[11pt, twocolumn]{article}
\usepackage{article}
\author{Torgeir Aamb\o \\ \textit{Norwegian Defence Research Establishment (FFI)}}
\title{A categorical formalization of epistemic uncertainty frameworks}
\date{}

\begin{document}
\maketitle

\begin{abstract}
    \textbf{Epistemic uncertainty arises in lack of complete knowledge about the state of a system. There are multiple mathematical frameworks for measuring such uncertainty quantitatively, often referred to as \emph{imprecise probability theories}. Inspired by work of Opdan, we introduce a general category theoretic definition of epistemic calculi, which we use as a foundation for modelling and studying contradictions and synergies between several philosophical epistemological concepts. We further develop an enriched category theoretic process for changing calculi, and use this to study relationships between existing examples, like possibility theory and certainty factors. Finally, we introduce a general categorical form of belief updating based on change of enrichment, and prove that Bayesian updating and possibilistic conditioning arise as examples.} 
\end{abstract}

\section{Introduction}

Any real world measurement contains some level of uncertainty, and it is often crucial to understand exactly how uncertain some measurement or statement is. Uncertainty can be either \emph{aleatoric} or \emph{epistemic}. The former relates to inherent randomness in systems, while the latter arises due to a lack of complete knowledge about the system in question -- it is this latter notion of uncertainty that is the focus of this paper. 

Our goal is to isolate the mathematical syntax of uncertainty away from its semantics, and define what we call \emph{uncertainty calculi} in the abstract. Inspired by \cite{opdan_2025}, where Opdan develops a categorical foundation for \emph{certainty factors}, as well as general $t$-norms in fuzzy logic \cite{zadeh_1965, klement-mesiar-pap_2000}, we base our theory on symmetric monoidal posetal categories, allowing us to more precisely study the algebraic and order-theoretic properties such uncertainty calculi exhibit. Together with certainty factors we formalize possibility theory, bipolar possibility theory and interval probabilities in this framework. 

Epistemology, the field concerned with the nature of knowledge, contains a plethora of philosophical differences of opinion, as well as a wide scape of theories. Using our formalization we model some of these -- \emph{closure}, \emph{strong conservativity}, \emph{fallibility}, \emph{cancellativity}, \emph{optimism}, \emph{scepticism} and \emph{idempotency} -- which allows us to both axiomatically compare different epistemic calculi on their philosophical positions and attitudes, as well as study interactions between these philosophical positions. 

\begin{introthm}
    If $\V$ is an uncertainty calculus, then $\V$ cannot be closed, conservative and cancellative simultaneously. Furthermore, if $\V$ is complete, then being closed is equivalent to being fallible. 
\end{introthm}

After considering these axiomatic properties of uncertainty calculi in isolation, we turn our attention to the relational aspect between different calculi. Using lax and op-lax monoidal functors $F\: \V \to \W$ we model \emph{conservative} and \emph{liberal} changes of calculi, that respectively cannot create or lose certainty via fusion. Using these we compare the four examples, proving for example that bipolar possibility theory and interval probability theory are equivalent as epistemic calculi. 

Having developed some understanding of the syntactic level, we use enriched category theory construct a general method for implementing epistemic calculi into systems where one wants to measure uncertainty -- the semantic level. This enrichment methodology is similar to \cite{bradley-terilla-vlassopoulos_2022}, where the authors develop an enriched category theoretic foundation for Large Language Models (LLM). This allows us to apply methods from \cite{kelly_1982} to model \emph{change of calculi} functors, which keeps the semantic layer constant and only changes the syntactic layer of uncertainty quantification. 

Bayesian updating, due to its non-symmetric nature, is not covered by the above theory. To fix this we formalize a general form of \emph{updating} in an epistemic calculus $\V$, and prove that these generalize familiar concepts. 

\begin{introthm}
    Let $\H$ be a hypotheses category enriched in an epistemic calculus $\V$. If $\V$ is the standard poset on $(0,\infty)$ together with multiplication, then $\V$-updating is the odds-form of Bayesian updating. If $\V$ is possibility theory, then $\V$-updating recovers the possibilistic conditioning of Dubois--Prade \cite{dubois-prade_1990}.
\end{introthm}

It is worth pointing out that our approach to studying the calculus of uncertainty stands heavily on the shoulders of giants. Our goal is not a final theory, but simply to contribute with a new perspective on already existing ideas, and to connect philosophical and mathematical ideas to categorical notions.

\section{Category theory}
Categories were introduced by Eilenberg--MacLane in \cite{eilenberg-maclane_1945} -- a standard reference being \cite{maclane_1998} -- and serves as a highly general and versatile structure for studying relational concepts throughout mathematics as a whole. A \emph{category} $\C$ consists of a class of \emph{objects} $\Ob(\C)$, and a class of \emph{morphisms} $\Mor(\C)$ between the objects. Every object has an identity morphism, and composition of morphism satisfies natural associativity and identity conditions. A classical example is the category of sets, which consists of all sets and their maps. 

A \emph{functor} $F$ between categories $\C$ and $\D$ assigns to any object in $\C$ an object in $\D$, and to any morphism in $\C$ a morphism in $\D$, such that the identity morphism and composition is preserved. 

We can also equip categories with extra structure. For example, a \emph{symmetric monoidal category} is a category $\C$ together with a binary operation 
\[-\otimes_\C - \: \C \times \C \to \C\] 
and a unit object $\1_\C$, satisfying natural unitality, associativity and symmetry conditions. We will often write this as a tuple $(\C, \otimes_\C, \1_\C)$. 

A functor $F$ between two symmetric monoidal categories $(\C, \otimes_\C, \1_\C)$ and $(\D, \otimes_\D, \1_\D)$ is said to be a \emph{symmetric monoidal functor} if there is a natural equivalence $F(x\otimes_\C y) \cong F(x)\otimes_\D F(y)$ for all objects $x,y \in \C$.

\subsection{Opdans CF category}
Opdan defines in \cite{opdan_2025} a category whose objects are real numbers in the interval $(-1,1)$, and whose morphisms are determined by the natural total order $\leq$. More precisely, for any two objects $x$ and $y$, i.e. two numbers in the interval $(-1,1)$, we have a single morphism $x\to y$ if $x\leq y$, and no morphisms otherwise. 

The interpretation is that an object in this category represents the belief strength in a statement; $0$ represents total epistemic uncertainty, $1$ represents true belief, and $-1$ represents true disbelief.  

This is given a symmetric monoidal structure via the hyperbolic sum. 
\[x\otimes y = \frac{x+y}{1+xy},\]
which makes it possible to combine, or fuse, beliefs. 

\begin{remark}
    The hyperbolic sum, also called the Einstein $t$-conorm, is used in physics to add relativistic velocities, see for example \cite{parker-jeynes-walker_2025}.
\end{remark}

Opdan writes the operation as
\[x\otimes y = \phi^{-1}(\phi(x)\phi(y)),\] 
where $\phi$ is the bijection $(-1,1) \to (0,\infty)$ given by $\phi(x)=\frac{1+x}{1-x}$. This makes it immediately obvious that the symmetric monoidal structure is both commutative and associative, due to the features of multiplication of real numbers. Furthermore, the monoidal unit is $0$, which makes sense intuitively, as adding information with total epistemic uncertainty should not change the belief value. 

Notice that if $x\leq x'$ and $y\leq y'$, then $x\otimes y \leq x'\otimes y'$, as can be seen by the monotonicity of $\phi$ and real number multiplication. This makes this category a \emph{symmetric monoidal poset} in the sense of \cite{fong-spivak_2019}.

Notice also that we can revoke Cromwell's rule -- originally stated by Lindley in \cite{lindley_1991} -- in this category by adding the extremes $-1$ and $1$, giving us the full closed interval $[-1,1]$. To make this work with the monoidal structure we declare that $-1\otimes 1 = 0$. 

We define $\CF$ to be the category described above, with the added extremes $-1$ and $1$. We now use this as the template for a general definition.

\section{Epistemic calculi}
Based on the properties of $\CF$ above, we can now make a general definition. An \emph{epistemic calculus} is a (non-trivial) unital symmetric monoidal posetal category $(\V, \leq_\V, \otimes_\V, \1_\V)$. We will usually omit the indices when the category is clear. We interpret $(\V, \leq, \otimes, \1)$ by objects in $\V$ being epistemic values, morphisms being comparisons of values, and the monoidal product being fusion of values -- the monoidal unit $\1$ is interpreted as total epistemic uncertainty.

\begin{remark}
    This definition is simply an abstraction of $t$-norms on the unit interval $[0,1]$, see \cite{klement-mesiar-pap_2000}. We also here restrict ourselves to symmetric structures, in order to model evidence fusion rather than updating, which we focus on in \cref{sses:updating}. 
\end{remark}

\subsection{Examples}
Building on the ideas in the previous section we describe some known examples in our abstract framework. 





\textbf{Possibility theory} \\
Possibility theory was introduced by Zadeh in \cite{zadeh_1978}, and later expanded upon further by Dubois--Prade, see \cite{dubois-prade_1988} or the survey \cite{dubois-prade_2015}. Possibility theory can be thought of as a version of probability theory where one replaces the standard ring structure on $\R$ with its tropical max-plus or min-plus semi-ring structure -- see \cite{green_1979} for an early description, and \cite{pin_1998} for a modernized approach. 

Possibility theory assigns a value $P(H)\in [0,1]$ to a given hypothesis $H$, which measures how possible the hypothesis is. A value of $0$ means that the hypothesis is not possible, while a value of $1$ means that it is fully possible. Fusion is given by $P(H_1\cup H_2)=\min(P(H_1), P(H_2))$. A second parameter, called the \emph{necessity}, defined by $N(H)=1-P(\neg H)$, measures the extent of how necessary the hypothesis is. 

Stripping away the connection to hypotheses and semantics, we define the category $\PT$ to be the standard preorder category on the unit interval $[0,1]$. We give this category the monoidal structure defined by 
\[x \otimes y = \min(x,y).\]
The monoidal unit is then $1$, which correctly equates complete epistemic uncertainty with pure possibility.  

Readers familiar with possibility theory might wonder where the necessity values enter the picture. However, as necessity requires the negation operator $\neg$ on hypotheses, it is necessarily a semantic operation, and thus not a part of the syntactic structure presented here. 

One way to include a type of necessity into the syntactic structure is by using bipolar possibility theory.

\textbf{Bipolar possibility theory} \\
Bipolar possibility theory, see \cite{benferhat-dubois-kaci-prade_2006} and \cite{dubois-lorini-prade_2013}, is an extension of possibility theory. It introduces a second set of operators, meaning it has four operators: $P$, $N$, $\Delta$ and $\nabla$. These are respectively called \emph{strong possibility}, \emph{weak necessity}, \emph{weak possibility} and \emph{strong necessity}. We focus here only on the strong ones. 

We define $\PTb$ to be the category with objects $\{(x, x')\in I\times I \mid x \leq x'\}$. This is interpreted as $x$ measuring the degree of rejection, or \emph{guaranteed impossibility} of a hypothesis and $x'$ measuring the degree of \emph{potential possibility}. The morphisms are determined by the preorder given by 
\[(x, x') \leq (y, y') \text{ if } x \geq y \text{ and } x'\leq y'.\]
We equip $\PTb$ with a symmetric monoidal structure defined by 
\[(x,x')\otimes (y,y') = (\max(x,y), \min(x',y')).\]
This is associative, symmetric and has unit $(0,1)$, which is the correct epistemic interpretation -- it is not rejected and fully possible, which is another way of stating complete epistemic uncertainty. 

Bipolar possibility theory allows us to retain the feature that belief and disbelief in a hypothesis are kept separate, and not treated equally -- disbelief is not simply the absence of belief, or, as popularized by Sagan, ``absence of evidence is not evidence of absence'' \cite{sagan_1997}. 

\textbf{Interval probabilities} \\
Another epistemic calculus comes from \emph{interval-based probabilities}, which is an interpretation of uncertainty in probabilities. Here we model a very simple version -- compared to for example Walley \cite{walley_1991} -- using the intersection operator as monoidal structure. 

Let $\IP$ be the symmetric monoidal category with objects the non-empty closed sub-intervals of $I = [0,1]$. The morphisms are determined by the containment-preorder, $[x,x']\leq [y,y']$ if $[x,x'] \subseteq [y, y']$, while the symmetric monoidal structure is given by 
\[[x,x']\otimes_\IP [y,y'] = [x,x']\cap [y, y'].\]

The monoidal unit is the full interval $[0,1]$. The interpretation is that interval probabilities are uncertain probabilities, and that the \emph{true} probability is residing somewhere in the interval. Larger intervals are more likely to contain this true value, and hence $\1 = [0,1]$ simply means that we know absolutely nothing about the true probability of the hypothesis.

\begin{remark}
    A natural candidate to add here would be Dempster-Shafer theory \cite{dempster_1967, shafer_1976}. However, it is currently not clear to us how this could be done in the presented framework, as the DS-theory operates a bit differently. We plan to investigate a categorical version of DS-theory in upcoming work. 
\end{remark}

\subsection{Extensions and properties}
With these examples in mind we can start to layer on extra structure to the general definition. We do this in the form of additional axioms one can add to an epistemic calculus. Note that these  are not something that is automatic from the general definition. 

\textbf{E1:} There exists an element $\top \in \V$ such that $x\leq \top$ for all $x$. 

An epistemic category satisfying \textbf{E1} is said to be \emph{optimistic}. If no such top element exists is said to be \emph{skeptic}. These definitions emulate a blend between the philosophical positions of \emph{epistemological optimism} and \emph{radical skepticism}, see \cite{haskins_2004, neto_1997}, and Cromwell's rule \cite{lindley_1991}. 

An extension of this is the requirement that $\V$ has all joins, which is a generalization of having all supremums. 

\textbf{E2:} Given a subset of objects $S\subseteq \V$, there is a least upper bound $\bigvee S$, called the \emph{join} of $S$. 

By \cite[2.72]{fong-spivak_2019} any epistemic calculus with all joins also has all \emph{meets}, making it a \emph{symmetric monoidal lattice}. An epistemic calculus satisfying \textbf{E2} is said to be \emph{complete}. If $\V$ is complete, then it also satisfies \textbf{E1}, as $\top$ is the meet over the empty set. The category $\CF$ is easily seen to be complete. 

In $\CF$ we notice that given objects $x,y$ such that $0 \leq x\otimes_\CF y$, then at least one of $x,y$ had to be non-negative as well. This is interpreted as as the impossibility of creating certainty from sources that didn't already contain certainty.

\textbf{E3:} If $\1_\V \leq x\otimes_\V y$, then either $\1_\V \leq x$ or $\1_\V \leq y$.

Furthermore, in $\CF$, note that we have an ``inverse'' to the symmetric monoidal product, called an \emph{internal hom}. An epistemic calculus $\V$ has an internal hom if there is a natural binary operation $[-,-]$ such that for all triplets of objects $x,y,z$ we have 
\[x\otimes y \leq z \iff x\leq [y,z].\]
For $\CF$ one can easily check that hyperbolic subtraction
\[[x,y] = \frac{x-y}{1-xy}\]
satisfies this property, making it an internal hom. This makes $\CF$ a \emph{closed symmetric monoidal preorder} in the sense of \cite{bradley-terilla-vlassopoulos_2022} and \cite{fong-spivak_2019}. 

\textbf{E4:} The symmetric monoidal structure is closed. 

\begin{remark}
    It is here important to note that \emph{closed} refers to the categorical structure, and not to \emph{epistemic closure}, as in for example \cite{holliday_2015}. 
\end{remark}

Having an internal hom is equivalent to the symmetric monoidal category being enriched over itself \cite[Section 1.6]{kelly_1982}. Interpreting the monoidal structure as a generalization of a $t$-norm, the internal hom is simply the generalization of a \emph{residuum}, see \cite{klement-mesiar-pap_2000}. Possibility theory is also closed via standard Gödel implication: 
\[[x,y] = 
\begin{cases}
    1, & x \leq y \\
    y, & x > y.
\end{cases}
\]

A property not satisfied by $\CF$ is \emph{strong epistemic conservativity}. This is a philosophical position essentially stating that it's reasonable to stick with your current beliefs unless you have good reasons to change them, \cite{mccain_2008}. This essentially means that we place the burden of proof on changing beliefs rather than holding them. Given an epistemic framework $\V$, we can model a very strong version of this as follows.

\textbf{E5:} For all $x\in \V$, we have $x\leq x\otimes y$ for any $y \in \V$.

Intuitively, given an epistemic value $x$ we can never reduce our belief in $x$ by fusing it with another epistemic value $y$. This is similar to modern digital echo-chamber belief dynamics. It is worth noting that the above form of conservativism is quite strong, and that more nuanced definitions depending on ther strength of the counter-evidence is possible to define.  

\begin{remark}
    None of the examples in this paper satisfy the \textbf{E5} axiom. However, the alternative version of possibility theory using the maximum operator instead of the minimum operator, i.e. the epistemic calculus $([0,1], \leq, \max, 0)$, \emph{is} strongly conservative.
\end{remark}


Another important epistemological concept, satisfied by for example $\PT$, is that of \emph{idempotency}. This is the position that ones belief in a statement should not be affected by the same evidence from the same source multiple times. If we learn evidence for a statement by reading a book, we cannot strengthen our belief in the statement by reading the same book again.  

\textbf{E6:} For all $x\in \V$, we have $x\otimes x = x$.

\begin{remark}
    An epistemic calculus satisfying \textbf{E6} avoids issues similar to Vogel's bootstrapping argument, \cite{vogel_2008}. 
\end{remark}

A final property that we consider here is that of \emph{epistemic fallibilism}, which is the philosophical position that all ones beliefs can be updated when presented with contradictory evidence or good new arguments for the contrary -- see for example \cite{kekes_1972} or \cite{feldman_1981}. Simply stated, no belief is infallible. We can model this as follows. 

\textbf{E7:} For all $x, y\in \V$ there exists a $z\in \V$ such that $x\otimes z \leq y$. If $\V$ satisfies \textbf{E1}, then we require this for all $x\neq \top$. 

Notice that when $x\leq y$ we can choose $z = \1$, so this is really a statement about updating a belief to obtain something one believes less strongly. 

Finally we add the property of \emph{cancellativity}, which states that fusing the same evidence with two values cannot change their relative strength. 

\textbf{E8:} If $x\otimes z \leq y \otimes z$, then also $x \leq y$. 

An epistemic calculus $\V$ being cancellative means that its algebraic structure behaves like an abelian group.

\subsection{Inconsistencies}
One benefit of defining epistemic uncertainty calculi in this general fashion is that one can prove general relationships between properties, and study which philosophical positions that are incompatible. 



One no-go result is the following, stating that an epistemic calculus $\V$ satisfying \textbf{E2} cannot satisfy both \textbf{E4} and \textbf{E5} and \textbf{E8}. The idea is that closure and cancellativity force negative beliefs to exist, which is incompatible with strong conservativity. 

\begin{theorem}
    \label{thm:no-go1}
    A complete non-trivial epistemic framework $\V$ cannot be both closed, strongly epistemically conservative and cancellative. 
\end{theorem}
\begin{proof}
    Assume for the sake of contradiction that $\V$ has all three properties, and choose objects $x,y$ such that $y \leq x$. The inner hom $[x,y]$ is the least upper bound of the set $\{z\in \V \mid x\otimes z \leq y\}$, as can be easily seen by its defining property. If $z$ is in this set, then as $\V$ is strongly conservative, we have 
    $x\otimes z \leq y \leq y\otimes z,$
    which implies $x\leq y$ by cancellativity. Hence, the set $\{z\in \V \mid x\otimes z \leq y\}$ is either empty, giving $[x,y] = \bot$ or contains only $\1$ if $x=y$. The first case implies that the defining formula $z\otimes x \leq y \iff z\leq [x,y] = \bot$ is only valid for $z = \bot$, which contradicts the non-triviality of $\V$. The second case also implies that all points in $\V$ are the same, contradicting non-triviality.
\end{proof}

We can also prove that a complete epistemic calculus satisfies \textbf{E4} if and only if it satisfies \textbf{E7}. 

\begin{theorem}
    A complete epistemic calculus $\V$ satisfies epistemic fallibility if and only if it is closed. 
\end{theorem}
\begin{proof}
    If we assume \textbf{E7}, we have that the set 
    \[\{z \mid x\otimes z \leq y\}\] 
    is non-empty. By completeness the least upper bound exists, and is equivalent to $[x,y]$ by the argument in the proof of \cref{thm:no-go1}. Conversely, if $\V$ is closed, then again the above set is non-empty, implying \textbf{E7}. 
\end{proof}

In Zadeh's fuzzy logic, see \cite{zadeh_1965} or \cite{klir-yuan_1995}, there is a standard result stating that the only idempotent $t$-norm is the Gödel $t$-norm (the $\min$ operator) -- see for example \cite[3.3]{dudek-jun_1999}. We can also prove an analogue of this result in our setting, making it even more structural. 

\begin{theorem}
    Let $\V$ be an epistemic calculus on a total order structure with $\1 = \top$. If the monoidal structure on $\V$ is idempotent, then $p\otimes q = \min(p, q)$. 
\end{theorem}
\begin{proof}
    By totality of the order assume without loss of generality that $x\leq y$. By monotonicity and idempotency of $\otimes$, we have 
    \[x = x\otimes x \leq x\otimes y \leq y\otimes y = y.\]
    Hence, $x=\min(x, y)\leq x\otimes y$. As $y\leq \1$ we further have 
    \[x\otimes y \leq x\otimes \1 = x,\]
    hence $x\otimes y = \min(x,y)$.     
\end{proof}

\begin{corollary}
    The category $\PT$ is the unique idempotent epistemic calculus on $[0,1]$ up to order-preserving isomorphism. 
\end{corollary}

We summarize this section with an overview of the properties of the examples presented. Note that all examples also have non-complete versions by invoking Cromwell's rule, and the list is easily extendable to these as well. 

\begin{table}[h]
\begin{tabular}{l | llllllll}
     & \textbf{E1} & \textbf{E2} & \textbf{E3} & \textbf{E4} & \textbf{E5} & \textbf{E6} & \textbf{E7} & \textbf{E8} \\
     \hline
$\CF$  &\checkmark  &\checkmark &\checkmark &\checkmark & & &\checkmark &\checkmark \\
$\PT$  &\checkmark  &\checkmark &\checkmark &\checkmark & &\checkmark &\checkmark & \\
$\PTb$ &\checkmark  &\checkmark &\checkmark & & &\checkmark &\checkmark & \\
$\IP$  &\checkmark  &\checkmark &\checkmark & & &\checkmark &\checkmark &
\end{tabular}
\end{table}

Enticingly, $\PTb$ and $\IP$ satisfy all of the same axioms. This is a good indication for strong relationships between these frameworks, and we will see shortly that these respective pairs are in fact isomorphic.

\section{Change of calculi}
Usually one is limited by using a single framework for measuring and quantifying uncertainty. One of the reasons for this work was to build a conceptual foundation for how one could \emph{change} ones framework during or after measurements are being made. It is perhaps easy to think that this is simply done by applying some translation function on values, but this makes it highly opaque how this affects the whole structure of uncertainty measurements. The language used above allows us to understand how to make sure everything works out nicely, and allows us to further understand the consequences of the changes.

If $\V$ and $\W$ are epistemic calculi, then a \emph{conservative change of calculi} is a lax symmetric monoidal functor $F\: \V \to \W$. In particular, $F$ satisfies 
\begin{itemize}
    \item $x\leq_\V y$, then $F(x)\leq_\W F(y)$;
    \item $F(x)\otimes_\W F(y) \leq_\W F(x\otimes_\V y)$;
    \item $\1_\W \leq_\W F(\1_\V).$
\end{itemize}  
Similarly, a \emph{liberal change of calculi} is an op-lax monoidal functor $F\: \V \to \W$, meaning it satisfies 
\begin{itemize}
    \item $x\leq_\V y$, then $F(x)\leq_\W F(y)$;
    \item $F(x\otimes_\V y) \leq_\W F(x)\otimes_\W F(y)$;
    \item $F(\1_\V) \leq_\W \1_\W.$
\end{itemize}  
The lax symmetric monoidality means that fusion of information along a conservative change cannot produce \emph{more} certainty than was already present in the original calculus. Conversely, fusion of information along a liberal change cannot produce \emph{less} certainty than was already present in the original calculus. If a functor is both conservative and liberal, we say it is \emph{balanced}. A balanced change then satisfies 
\begin{itemize}
    \item $F(x)\otimes_\W F(y) = F(x\otimes_\V y)$;
    \item $\1_\W = F(\1_\V),$
\end{itemize}  
in other words, it is a symmetric monoidal functor.



Let us prove the statement we made earlier, about the equivalence between $\PTb$ and $\IP$. 

\begin{lemma}
    The assignment 
    \begin{align*}
        F \: \PTb &\to \IP \\
        (x,x') &\longmapsto [x,x']
    \end{align*}
    is a balanced change of calculi. In particular, with the intersection monoidal structure on $\IP$, the map $F$ is an equivalence. 
\end{lemma}
\begin{proof}
    The map is easily seen to be an equivalence on objects. It is also monotone on the monoidal structures. Finally, for two objects $(x, x')$ and $(y, y')$ we have 
    \begin{align*}
        F((x, x')\otimes (y, y')) 
        &= F(\max(x, y), \min(x',y')) \\
        &= [\max(x, y), \min(x',y')] \\
        &= [x, x']\cap [y, y] \\
        &= F(x, x')\otimes F(y, y'),
    \end{align*}
    which proves it is balanced. 
\end{proof}

\begin{remark}
    Note that while the above map is an equivalence, the use and semantics are often different. The degree of rejection is here mapped to the lower bound for the probability, which intuitively seems out of place. This shows that the underlying calculi might be equivalent, even though our interpretation of them are different. We are only concerned with this type of syntactic equivalence in the present work, but hope to expand this notion further in the future to also include interpretation into the structure. 
\end{remark}

We also have a natural map $F\: \PT \to \CF$ given on objects by $F(x) = 2x-1$. As this map is monotone, it also extends to a functor. Notice that we have $F(1) = 1 \geq 0$. Also, assuming without loss of generality that $x \leq y$, then 
\[F(x\otimes_\PT y) = F(x) \leq F(x)\otimes_\CF F(y)\]
as you cannot create lower epistemic values in $\CF$ by fusing with something of higher value. Hence $F$ is an op-lax monoidal functor, or in other words a liberal change of calculi. 

This map feels a bit odd, as the monoidal unit $1$ is mapped to the top value $\top$, and since it is a rescaling, $\frac{1}{2}$ is mapped to $0$. This is ok formally, but perhaps not a good map to compare epistemic frameworks. The problem is that our version of possibility cannot incorporate negative evidence, and needs semantics to have a notion of necessity, which is why we added bipolar possibility theory. Let us see how this also fixes the above problematic map. 

Recall that a point $(x,y) \in \PTb$ is interpreted as $x$ being the degree of rejection, while $y$ is the degree of possibility. We define 
\begin{align*}
    F \: \PTb &\to \CF \\ 
    (x,y) &\longmapsto y-(1-x)
\end{align*}
which measures possibility minus non-rejectedness, resulting in a kind of single value belief. In particular, $F(0,1)=0$, $F(1,1)=1$ and $F(0,0) = -1$. One can check that this gives a liberal change of calculi. 

\begin{remark}
    Notice that the ``inconsistent'' state $(1,1)$, interpreted as both fully possible and fully impossible due to rejection, is sent to the maximal belief state. This is also not very intuitive. One way to fix this is to interpret the pair $(x,y)$ as $x$ being non-rejectedness, and $y$ being possibility, and changing the monoidal structure to $\min$-$\min$, instead of $\max$-$\min$. Then one can define a conservative change $F(x,y)=x+y-2$, which has a more suitable semantic interpretation. 
\end{remark}

\subsection{Enriched category theory}
One way to add epistemic uncertainty quantification directly into a process is by utilizing \emph{enriched categories}. This is a theory that endows a category with extra structure, coming from another category. Our approach is \emph{relational} and not pointwise. This follows the dictums of category theory, where the relational structure and how an object is situated in its context, is more important than the object itself. 

Let $(\V,\otimes_\V, \1_\V)$ be a symmetric monoidal category. A $\V$\emph{-enriched category} $\C$ is a collection of objects $\Ob(\C)$, together with an object $\V(x,y) \in \V$ for any pair of objects $x, y \in \C$, a morphism $\1_\V \to \V(x,x)$ called the identity, and a morphism 
\[\V(y,z)\otimes_\V \V(x,y) \to \V(x,z)\]
for any three objects $x, y, z$ called the composition. These must also satisfy natural associativity and unitality conditions. Similarly, there is a notion of a $\V$-enriched functor, which assigns objects to objects and morphism $\V$-objects to morphism $\V$-objects, see the standard reference \cite{kelly_1982} for details. The category of $\V$-enriched categories and $\V$-enriched functors is denoted $\Cat_\V$. 

The central idea for adding epistemic uncertainty to a category $\C$ is to enrich it in an epistemic calculus. For example, define $\H$ to be a category consisting of hypotheses $H$, and morphisms $H\to H'$ being belief updates. We can make $\H$ into an enriched category over any uncertainty framework $\C$ in order to have a quantitative measure of how much we believe hypothesis $H'$ based on $H$. In particular this means assigning to each belief update $f\: H\to H'$ an epistemic value $\C(f)$, such that $\C(\id_H) = \1_\C$ for any $H$ --  no hypothesis can justify its own belief -- and such that composition respects the monoidal structure in $\C$, i.e. 
\[\C(f'\circ f) = \C(f)\otimes_\C \C(f')\]
for any $H\to H'\to H''$. In \cite{opdan_2025}, Opdan uses these ideas to model sensor systems for submarine detection, and uses $\CF$ as his enriching category. In \cite{bradley-terilla-vlassopoulos_2022} the authors use these ideas to model large language models, using $([0,1], \leq, \cdot, 1)$ as their epistemic calculus. 

Due to the general nature of our setup, we now immediately get a procedure for changing how a category uses an epistemic calculus. 

\begin{theorem}[\cite{eilenberg-kelly_1966}]
    For any conservative change of calculi $F\:\V \to \W$, there is an induced functor 
    \[F^*\:\Cat_\V \to \Cat_\W,\]
    which assigns to a $\V$-enriched category $\C$ the $\W$-category consisting of the same objects, but whose morphism $\W$-objects are determined by the images of the morphism $\C$-objects under $F$. Similarly, a liberal change of calculi induces a functor in the opposite direction. 
\end{theorem}

This means that for any category $\C$ enriched in some epistemic calculus $\V$, it is possible to change framework in a very natural way. The semantics stay the same, but the syntax changes. 

\begin{remark}
    It would be interesting to use this to study LLMs via the enriched syntax and semantic categories from \cite{bradley-terilla-vlassopoulos_2022}, and use alternative uncertainty calculi. As discussed in \cite{opdan_2025}, one could then possibly avoid issues with posterior collapse. 
\end{remark}

\subsection{Enriched epistemic updating}
\label{sses:updating}
One piece of the puzzle that has been missing from this story so far is how Bayesian updating, which is perhaps the standard way of encoding and updating uncertainties, fits into this framework. 

First, notice that Bayes rule is not symmetric monoidal, hence it does not simply define an epistemic calculus. This also makes sense intuitively. The underlying mathematical logic of Bayes rule is not Bayes rule itself, but \emph{multiplication} and \emph{division}. Bayes rule, and other updating mechanisms for beliefs, lie one level above the mathematical logic. Let us describe this, and derive Bayes rule from a general categorical process. 

\begin{remark}
    Note that there are other categorical theories for Bayesian updating, like Markov categories, see \cite{fritz_2020} and \cite{cho-jacobs_2019}. 
\end{remark}

Let $\V$ be a closed epistemic calculus, with monoidal structure $-\otimes -$ and internal hom $[-,-]$. Let furthermore $\H$ be a category of hypotheses, that is enriched in $\V$. The interpretation is again that for two hypotheses $H$ and $H'$, the $\V$-object $\H(H,H')$ describes our belief in $H'$ relative to $H$. 

We want to understand what happens when we incorporate a \emph{new} piece of evidence. One can incorporate the notion of ``new'' in several different ways, but for the moment we just let $E\in \H$ . One can also think about $E$ coming from some evidence category via a $\V$-functor into $\H$.  

Now, given two hypotheses $H$ and $H'$, and some evidence $E$, we define the $\V$-\emph{update} of $\H$ based on $E$, to be the same category $\H$, but with enriched structure given by 
\[\H_E(H,H') := [\H(H,H')\otimes \H(H',E), \H(H,E)].\]
This is in fact an enrichment, as can be seen by currying and internal hom composition. The interpretation is that $\H_E(H,H')$ measures the compatability of $\H(H,H')$ with the evidence $E$. 

The general shape of this idea is that ``updated comparisons = evidence - (prior comparison + evidence explained by the comparison)'', at least when interpreted in the correct coordinates. We can recognize Bayes formula to be of this shape, a statement we can make formal as follows. 

\begin{theorem}
    \label{thm:bayes-rule}
    Let $\V$ be the epistemic calculus defined by the standard preorder on $(0,\infty)$, representing likelihood ratios. We equip it with a monoidal structure $\otimes$ given by multiplication, which makes the internal hom division. If $\H$ is a $\V$-category of hypotheses, with $\V$-structure given by $\H(H, H')=\frac{p(H')}{p(H)}$, then $\V$-updating is non-normalized Bayesian updating. 
\end{theorem}
\begin{proof}
    Let $E$ be an object of $\H$, and let $\H(H, E) = p(E\mid H)$. By definition we have
    \begin{align*}
        \H_E(H,H') 
        &= \frac{p(E\mid H)}{\frac{p(H')}{p(H)}\cdot p(E\mid H')} \\
        &= \frac{p(H) \cdot p(E\mid H)}{p(H')\cdot p(E\mid H')} \\
        &= \frac{p(H\mid E)}{p(H'\mid E)}
    \end{align*}
    which is precisely the posterior odds. 
\end{proof}

One can also normalize the above formula, but this requires chosing a reference object. If $\V$ is complete, then it is natural to use $\top \in \V$. Doing this gives a categorical construction of absolute probabilities. 

Let us also look at the $\V$-updating in some other examples. For $\V = \PT$ we recover a relational version of Dubios--Prade's possibilistic conditioning, where $\H_E(H,H')$ is given by
\[
\begin{cases}
    1, &\min(\H(H,H'), \H(H',E)) \leq \H(H,E) \\
    \H(H,E),  &\min(H,H'), \H(H',E) > \H(H,E).
\end{cases}
\]
The standard absolute possibilistic conditioning, see \cite{dubois-prade_1990}, can be recovered from the relational one, by defining $p(H) = \H(R,H)$ for some reference hypothesis $R$. 

For certainty factors we can convert the general form of the $\CF$-update by using evidence-coordinates, i.e., applying $\tanh^{-1}$ to the expression, which gives: 
\[e_E(H,H') = e(H,E) - (e(H,H')+e(H',E)),\]
where $e(H,H')= \tanh^{-1}\H(H,H')$ is the evidence. To get back the certainty factor we simply compute 
\[\H_E(H,H')= \tanh(e_E(H,H')).\]
This means that $\CF$-updating is similar to log-odds Bayesian updating, but with a bounded non-linear back-projection to $[0,1]$. We plan to study this general $\V$-updating further in future work. 

\section{Conclusion}
In this paper we have introduced a general category theoretic foundation for the syntax of epistemic uncertainty, which cements uncertainty quantification into a common framework on a higher level than what is usually done. Interpreting known examples in this framework shows that it is versatile, and allows for a new perspective on existing examples. 

Our approach allowed us to naturally construct mathematical formalizations of epistemic philosophical positions and concepts, and to have a natural way to compare them. We also described which philosophical positions and concepts hold for each of these examples. 

We then introduced a structured approach for comparing epistemic calculi, together with an understanding of how such changes affects the logic and epistemic quality. This led to a formal method for assigning epistemic uncertainty to a given structure -- in the form of a general category. As category theory is very general, and allows for the modeling of a lot of structures, this gives a robust and mathematically sound framework for discussing the uncertainties assigned to it, and how it combines along features of the structure. This method also allowed us to construct a mathematically rigorous process for changing the epistemic calculus assigned to a system. 

Finally we introduced generalizations of belief updating based on general epistemic calculi, and proved that Bayesian updating and possibilistic conditioning fall out as examples. 

There are bound to be other positions and concepts that can be modeled in this framework, and we hope that this paper can be a starting point to a broader discussion and exploration on categorical modeling of epistemic ideas.

\printbibliography{}  
\end{document}